\documentclass[12pt,article]{amsart}
\usepackage[margin=1in]{geometry}
\usepackage{amsmath, amssymb,amsthm,url,mathrsfs,graphicx,amscd,amsfonts}
\usepackage[mathscr]{eucal}
\usepackage{dsfont}
\usepackage{mathtools}
\usepackage{hyperref}
\usepackage[normalem]{ulem} 
\usepackage[utf8]{inputenc}
\usepackage{graphicx,todonotes,hyperref}
\usepackage{epstopdf}
\usepackage{inputenc}
\usepackage{enumitem}
\usepackage{amsmath}
\usepackage{graphicx,todonotes,hyperref}
\usepackage{xcolor}
\usepackage{amssymb}

\newtheorem*{acknowledgements*}{Acknowledgements}

\newtheorem{prob}{Open Problem}

\allowdisplaybreaks
\newtheorem{theorem}{Theorem}[section]
\newtheorem{lemma}[theorem]{Lemma}
\newtheorem{corollary}[theorem]{Corollary}
\newtheorem{proposition}[theorem]{Proposition}
\newtheorem*{definition}{Definition}
\newtheorem*{question}{Question}
\newtheorem*{theorem*}{Theorem}
\theoremstyle{remark}
\newtheorem{remark}{Remark}

\newtheorem{case}{Case}

\numberwithin{equation}{section}
\newtheoremstyle{named}{}{}{\itshape}{}{\bfseries}{.}{.5em}{\thmname{#1 }\thmnote{#3}}
\theoremstyle{named}
\newtheorem*{namedthm}{Theorem}

\newcommand{\R}{\mathbb{R}}
\newcommand{\Z}{\mathbb{Z}}
\newcommand{\N}{\mathbb{N}}
\newcommand{\boldX}{\boldsymbol{X}}
\newcommand{\boldx}{\boldsymbol{x}}
\newcommand{\boldy}{\boldsymbol{y}}
\newcommand{\boldk}{\boldsymbol{k}}
\newcommand{\boldn}{\boldsymbol{n}}
\newcommand{\boldj}{\boldsymbol{j}}
\newcommand{\boldt}{\boldsymbol{t}}
\newcommand{\boldu}{\boldsymbol{u}}
\newcommand{\boldz}{\boldsymbol{z}}
\newcommand{\boldh}{\boldsymbol{h}}
\newcommand{\boldg}{\boldsymbol{g}}

\newcommand{\bcalX}{\boldsymbol{\mathcal{X}}}

\newcommand{\boldalpha}{\boldsymbol{\alpha}}
\newcommand{\boldmu}{\boldsymbol{\mu}}
\newcommand{\boldgama}{\boldsymbol{\gamma}}

\begin{document}
	
	\title[PPC for higher dimensional real sequences]{Poissonian Pair Correlation for Higher Dimensional Real Sequences}
	
	\author{Tanmoy Bera, Mithun Kumar Das, Anirban Mukhopadhyay}
	\address{$^{1,3}$The Institute of Mathematical Sciences, A CI of 
		Homi Bhabha National Institute, CIT Campus, Taramani, Chennai
		600113, India.}
	\address{$^{2}$ National Institute of Science Education and Research, A CI of 
		Homi Bhabha National Institute, Jatni, Khurda, 
		752050, India\\ and \\ Mathematics Section, The Abdus Salam International Centre for Theoretical Physics, Str. Costiera, 11, 34151 Trieste, Italy}
	\email[Tanmoy Bera$^1$]{tanmoyb@imsc.res.in}
	\email[Mithun Kumar Das$^2$]{das.mithun3@gmail.com} 
	\email[Anirban Mukhopadhyay$^3$]{anirban@imsc.res.in}

	\maketitle
	
	\begin{abstract}
		In this article, we examine the Poissonian pair correlation (PPC) statistic for higher-dimensional real sequences. Specifically, we demonstrate that for $d\geq 3$, almost all $(\alpha_1,\ldots,\alpha_d) \in \mathbb{R}^d$, the sequence $\big(\{x_n\alpha_1\},\dots,\{x_n\alpha_d\}\big)$ in $[0,1)^d$ has PPC conditionally on the additive energy bound of $(x_n).$ This bound is more relaxed compared to the additive energy bound for one dimension as discussed in \cite{CA2021Real}. More generally, we derive the PPC for $\big(\{x_n^{(1)}\alpha_1\},\dots,\{x_n^{(d)}\alpha_d\}\big) \in [0,1)^d$ for almost all $(\alpha_1,\ldots,\alpha_d) \in \mathbb{R}^d.$ As a consequence we establish the metric PPC for $(n^{\theta_1},\ldots,n^{\theta_d})$ provided that all of the $\theta_i$'s are greater than two.
	\end{abstract}
	
	\section{Introduction}
	A sequence of real numbers $(a_k)$ in $[0,1)$ is called uniformly distributed if for any subinterval $[a, b$) in $[0,1),$
	\[\lim_{n\rightarrow \infty}\frac{\# \{k\leq n: a_k\in [a, b)\}}{n}=b-a.
	\]
	A classic example of such a sequence is the Kronecker sequence, defined as $a_k=\{k\alpha\}$ (where $\{\cdot\}$ denotes the fractional part). This sequence exhibits uniform distribution if and only if $\alpha$ is irrational. The concept of uniform distribution of a sequence signifies that the relative frequency of elements falling into any given subinterval asymptotically approaches the length of that subinterval. 
	
	When we establish that a sequence follows uniform distribution, it naturally prompts us to explore its statistical characteristics on more refined scales. One such statistical property of interest is the pair correlation. For $s>0$, the pair correlation function is defined as
	\[R_{2}(s, a_n, N):=\frac{1}{N}\#\Big\{1\leq n\neq m\leq N: \|a_n-a_m\|\leq\frac{s}{N}\Big\},\]
	where $\displaystyle\|y\|=\min_{k\in\Z}|y+k|$. We say that $(a_n)$ in $[0,1]$ exhibit Poissonian pair correlation if for all $s>0$
	\begin{align*}
		\lim_{N\rightarrow \infty}R_{2}(s, a_n, N)=2s.
	\end{align*}
	This topic has gained significant prominence over the past three decades, following the pioneering work of Rudnick and Sarnak~\cite{rudnick1998pair}. While substantial progress has been made in developing theory for one-dimensional pair correlation, our understanding remains limited when it comes to higher-dimensional sequences.  At the same time the theory of uniform distribution is well studied in higher dimensions (see the book~\cite{DT}).
	
	Recently, the concept of pair correlation in higher-dimensional sequences has been explored in \cite{hinrichs2019multi,steinerberger2020poissonian,BDM}. This exploration involved constraining the test sets to have their centers at the origin with respect to both the sup-norm and the $2$-norm. A natural question is whether the pair correlation property in higher-dimensional sequences is dependent on the choice of norms.
	
	For higher-dimensional real sequences, the pair correlation function can be defined as 
	\begin{definition}
		Let $d\geq 1$ be a natural number. Let $s>0$ be a real number and $N$ be a natural number, $(\boldy_n)\subseteq[0,1)^d$ be a sequence and $\|\cdot\|$ be a norm in $\R^d$. The $d$-dimensional pair correlation function $R_{2,\|\cdot\|}^{(d)}(s,(\boldy_n),N)$ is defined as follows:
		\[R_{2,\|\cdot\|}^{(d)}(s,(\boldy_n),N):=\frac{1}{N}\#\Big\{1\leq n\neq m\leq N: \|\boldy_n-\boldy_m\|^{\text{(intdist)}}\leq\frac{s}{N^{1/d}}\Big\},\]
		where $\displaystyle\|\boldy\|^{\text{(intdist)}}=\min_{\boldk\in\Z^d}\|\boldsymbol{y}+\boldk\|.$\\
		\noindent
		We say that $(\boldy_n)$ has Poissonian pair correlation with respect to $\|\cdot\|$-norm (or, $\|\cdot\|$-PPC) if for all $s>0$,
		\begin{align}\label{defn}
			R_{2,\|\cdot\|}^{(d)}(s,(\boldy_n),N)\to \omega_{(d,\|\cdot\|)}s^d \quad \mbox{ as } N\rightarrow \infty,
		\end{align}
		where $\omega_{(d,\|\cdot\|)}$ is the volume of the unit ball in $(\R^d,\|\cdot\|).$
	\end{definition}
	
	In this article we study the pair correlation function with respect to the sup-norm. We say $(\boldy_n)$ has $\infty$-PPC if~\eqref{defn} holds for sup norm. In this case the pair correlation function is denoted by  $R_{2,\infty}^{(d)}(s,(\boldy_n),N).$
	
	\begin{definition}[Metric Poissonian pair correlation (MPPC)]
		We say that a sequence $(\boldx_n)$ in $(\R_{>0})^d$  possesses $\infty$-MPPC if the sequence $(\{\boldx_n\boldalpha\})=(\{x_n^{(1)}\alpha_1\}, \ldots, \{x_n^{(d)}\alpha_d\})$ exhibits $\infty$-PPC for almost all $\boldalpha \in \mathbb{R}^d$.
	\end{definition}
	
	For one dimensional real sequence, the additive energy for the set of first $N$ elements is defined as 
	\[E(X_N):= \#\{\boldn\in[1,N]^4: |x_{n_1}+x_{n_2}-x_{n_3}-x_{n_4}|<1\}.\]
	It is worth noting that for integer sequences, the condition `$|x_{n_1}+x_{n_2}-x_{n_3}-x_{n_4}|<1$' in $E(X_N)$ reduces to `$x_{n_1}+x_{n_2}=x_{n_3}+x_{n_4}$'.
	For one-dimensional integer sequences, the study of metric pair correlation has been explored in \cite{CA2017additive, Bloom2019GCDSA} conditionally on additive energy bound.
	In the case of one-dimensional real sequences, a recent and highly influential set of results has been established by Aistleitner, El-Baz, and Munsch in \cite{CA2021Real}. Their contributions are as follows:
	\begin{namedthm}[A]
		Let $(x_n)$ be a sequence of positive real numbers for which there exists a constant $c>0$ such that $x_{n+1}-x_n\geq c,$ for all  $n\geq 1$. Assume that there exists some $\delta>0$ such that,
		\[E(X_N)\ll N^{\frac{183}{76}-\delta}.\] 
		Then $(x_n)$ has MPPC.
	\end{namedthm}
	Furthermore, they consider the quantity $E_\gamma(X_N)$ for $\gamma\in (0,1]$, given by
	\begin{align}\label{variant real additive energy}
		E_\gamma(X_N)=\#\{\boldn\in[1,N]^4: |x_{n_1}+x_{n_2}-x_{n_3}-x_{n_4}|<\gamma\}.
	\end{align}
	Conditioning on this quantity, they have established the following result:
	\begin{namedthm}[B]\label{Theorem B}
		Let $(x_n)$ be a sequence of positive real numbers for which there exists a constant $c>0$ such that $x_{n+1}-x_n\geq c,$ $n\geq 1.$
		Assume that there exists some $\delta>0$ such that for all $\eta>0$ we have,
		\[E_\gamma(X_N)\ll_{\eta,\delta} N^{2+\eta}+\gamma N^{3-\delta}\]
		as $N\to\infty,$ uniformly for $\gamma\in(0,1].$ Then the sequence $(x_n)$ has MPPC.
	\end{namedthm}
	Clearly, this enables us to deduce metric pair correlations in certain cases where the condition on additive energy in Theorem A does not hold. As an application of Theorem B, in conjunction with a result by Robert and Sargos ( see Theorem~\ref{RS theorem}), it has been demonstrated that the sequence $(n^\theta)$ with $\theta>1$ exhibits the MPPC property (see \cite[Theorem 3]{CA2021Real}). For $(n^\theta)$ with $0<\theta<1$, the MPPC property has been established in \cite{RT}.
	
	For higher dimensional integer sequences, Hinrichs, Kaltenb{\"o}ck, Larcher, Stockinger and Ullrich~\cite{hinrichs2019multi} considered the $d$-dimensional sequence $(\{x_n\boldalpha\})$ with $(x_n) \subset\mathbb{N}$. They showed that if for any $\epsilon>0$
	\begin{align*}
		E(X_N)\ll \frac{N^3}{(\log{N})^{1+\epsilon}},
	\end{align*}
	then $(\{x_n\boldalpha\})$ has $\infty$-PPC for almost all $\boldalpha\in\R^d.$
	
	In a recent publication~\cite{BDM}, the authors established the $\infty$-MPPC property for any $d$-dimensional integer sequence $\big(x_n^{(1)},\dots,x_n^{(d)}\big)$ with strictly increasing components. This result is conditional upon the joint additive energy bound $E(\boldX_N^D) \ll N^{3-\delta}$ for any $\delta>0$, where $D:=\{1,\ldots,d\}$ and the joint additive energy $E(\boldX_N^D)$ is defined as:
	\begin{align*}
		E(\boldX_N^D):=\#\Big\{\boldn\in[1,N]^4: x_{n_1}^{(l)}+x_{n_2}^{(l)}=x_{n_3}^{(l)}+x_{n_4}^{(l)},\: 1\leq l\leq d\Big\}.
	\end{align*}
	In order to obtain our results for higher dimensional real sequences, it is necessary to define the analogous version of the joint additive energy for real sequences.  
	\begin{definition}(Joint additive energy for real sequences)
		Let $\boldX_N=\{(x_n^{(1)},\dots,x_n^{(d)})\in \mathbb{R}^d : n\leq N\}$ and  $\boldsymbol{\gamma}\in (0,1]^d$ be real numbers. For $D'\subseteq D:= \{1,\ldots,d\}$ we define the quantity  $E_{\boldsymbol{\gamma}}(\boldX_N^{D'})$  by
		\begin{align}\label{varient of joint additive energy}
			E_{\boldsymbol{\gamma}}(\boldX_N^{D'})=\#\{\boldn\in[1,N]^4: |x_{n_1}^{(l)}+x_{n_2}^{(l)}-x_{n_3}^{(l)}-x_{n_4}^{(l)}|<\gamma_l,\: l\in D'\}.
		\end{align}
		The joint additive energy for higher-dimensional real sequences, denoted as $E(\boldX_N^D)$ (with notation consistent with the integer case), is defined by $E(\boldX_N^D) := E_{\boldsymbol{1}}(\boldX_N^D)$.
	\end{definition}
	Before we state our general result, we state a special case.
		\begin{theorem}\label{thmoc}
			Let $d\geq3$ and $(x_n)$ be a sequence of positive real numbers satisfying the spacing condition $x_{n+1}-x_n\geq c,$ for some $c>0.$ Assume that there exists some  $\lambda>0$ such that
			\[E(X_N)\ll N^{\frac{183}{76}+\frac{16}{76}\left(1-\frac{1}{d}\right)-\lambda}.\]
			as $N\to\infty.$ Then $(\{x_n\boldalpha\})$ has $\infty$-PPC for almost all $\boldalpha\in\R^d.$
		\end{theorem}
		
		We observe that when $d=1$, the exponent of the additive energy bound is at most $\frac{183}{76}\approx 2.4079$. This coincides with the exponent found in the result by Aistleitner et al.~\cite{CA2021Real} (see Theorem A). However, as soon as we consider $d\geq 3$, we can relax the additive energy bound by an additional quantity of $\frac{4}{19}\left(1-\frac{1}{d}\right)$ in the exponent, and as $d$ increases, the bound can go up to a maximum of $\frac{199}{76}\approx 2.6184$.

		Regarding the case when $d=2$, we have not been able to prove this stronger version, but we establish a weaker result in Theorem \ref{thm1}.
		\begin{theorem}\label{thm1}
			Let $(\boldx_n)$ be a sequence in $\R^d$ with positive components and $d\geq 2$. Also assume that there exist some $c>0$ such that $x_{n+1}^{(l)}-x_n^{(l)}\geq c$ for all $1\leq l\leq d$ and $n\geq 1$. If for all $D'\subseteq D=\{1,\ldots,d\}$ of $d'$ elements satisfy the bound
			\begin{align}\label{thm1 conditions}
				E(\boldX_N^{D'})\ll N^{4-\frac{1}{d}\left(2d'-\frac{31}{76}\right)-\lambda_{D'}},
			\end{align}
			for some $\lambda_{D'}>0$, then $(\boldx_n)$ has $\infty$-MPPC.
		\end{theorem}
		\begin{remark}
			Conditionally on the Lindelöf hypothesis, the upper bound of $E(X_N)$ in Theorem \ref{thmoc} can be relaxed to $E(X_N)\ll N^{3-\lambda}$ for some $\lambda > 0$. This assumption essentially coincides with the integer case. For Theorem \ref{thm1}, under the Lindelöf hypothesis, the upper bound of $E(\boldX_N^{D'})$ can be relaxed to $E(\boldX_N^{D'})\ll N^{4-\frac{1}{d}(2d'-1)-\lambda_{D'}}$ for all $D' \subseteq D$.
		\end{remark}
		Now we consider the higher-dimensional analogue of the sequence $(n^\theta).$ If the $\theta_i$'s are integers, this case has already been addressed in \cite{BDM}. Let $\theta_1, \dots, \theta_d>2$ be real numbers that are not integers. 
		As a consequence of Theorem~\ref{thm1} and the joint additive energy bound~\eqref{ntheta add eng} we get the following.
		\begin{corollary}\label{high ntheta cor}
			If all of $\theta_1, \dots, \theta_d$ are greater than 2, then the sequence $(n^{\theta_1}, \dots, n^{\theta_d})$ has $\infty$-MPPC.
		\end{corollary}
		
		In a recent study, Hofer and Kaltenb\"ock~\cite{haltonnotPPC2021} raised the following question:
		
		\begin{question}
			In the context of Poissonian pair correlations in higher-dimensional settings, is it necessary for all the component sequences to exhibit Poissonian pair correlations?
		\end{question}
		
		In our work \cite{BDM}, we provide a negative answer to this question when all the components of a higher-dimensional sequence are integers. The following corollary further offers a negative response to this question when the components are real.
		\begin{corollary}\label{n,logn PPC}
			For $A\geq1$, the $2$-dimensional sequence $(n,n\log^A{ n})$ has $\infty$-MPPC.
		\end{corollary}
		This corollary follows from Theorem~\ref{thm1}, Corollary~\ref{boundcor},
		and the additive energy bound of $(n\log^An)$ in \eqref{additive of nlogn}. Note that $(\{n\alpha\})$ for any $\alpha\in \mathbb{R}$ does not have PPC and the same for $(\{n(\log{n})^A\alpha\})$ is not known.
		
		\subsection*{Strategy}The proofs of Theorem~\ref{thmoc} and Theorem~\ref{thm1} rely on employing higher-dimensional Fourier analysis on the torus $\mathbb{T}=[0,1]^d$ to detect the indicator function corresponding to the pair correlation function $R_{2,\infty}^{(d)}(s,(\boldsymbol{y}_n),N)$. Subsequently, we exploit the property of the $\infty$-norm to derive majorant and minorant functions for the indicator function, which take the form of truncated trigonometric series. The key insight here is to demonstrate that the variance of these majorant and minorant functions with respect to a specific measure on $\mathbb{R}^d$ becomes negligible as $N$ grows sufficiently large.
		This idea essentially extends the approach found in \cite{CA2021Real}, which deals with the one-dimensional case.
		
		By taking advantage of the analysis with respect to the sup-norm for the sequence of the form $(\{x_n\alpha_1\},\ldots,\{x_n\alpha_d\})$, we simplify the counting problem from a system of $2d$ variables to a system of $d$ variables (refer to \eqref{same component} and \eqref{reduction}). As a result,  Theorem \ref{thmoc} establishes the MPPC conditionally on the additive energy of the  sequence $(x_n)$ with a more relaxed additive energy bound than the one presented in \cite{CA2021Real}.
		
		Once we establish Theorem~\ref{thm1}, the proof of Corollary~\ref{high ntheta cor} and Corollary~\ref{n,logn PPC} reduces to demonstrating that the joint additive energy bound of those sequences is at most $N^{2+\epsilon}.$ To achieve this, we make use of the work of Watt~\cite{N Watt} on counting solutions to a specific system of Diophantine inequalities, coupled with the Van der Corput estimate of the exponential sum. The remainder of the proof involves a generalization of the argument found in the proof of \cite[Theorem 3]{CA2021Real}. 
		
		We list some of the many interesting open problems related to our results in this paper.
		\begin{prob}
			Prove or disprove that the notion of pair correlation in higher dimension depends on the norm.
		\end{prob}
		
		\begin{prob}
			Show that Theorem \ref{thmoc} remains valid under the weaker additive energy bound assumption
			$E(X_N)\ll N^{3-\lambda }$ for some $\lambda >0$.
		\end{prob}
		\begin{prob}
			Let $0<\theta_1,\cdots,\theta_d<2$ be real numbers not all integers. Show that $(n^{\theta_1},\dots,n^{\theta_d})$ has $\infty$-MPPC.
		\end{prob}
		
		\section{Estimating joint additive energy for real sequences}\label{additive energy section}
		This section is devoted to obtain an upper bound of joint additive energy for the sequence $(n^{\theta_1},\dots,n^{\theta_d})$ and $(n,n\log^A n)$, where $A\geq1$. Such bounds are used to obtain 
		Corollary~\ref{high ntheta cor} and Corollary \ref{n,logn PPC}.
		We recall a result by Robert and Sargos \cite{Robert-Sargos} for the one-dimensional case:
		\begin{theorem}\cite[Theorem 2]{Robert-Sargos}\label{RS theorem}
			Let $\theta\neq 0,1$ be a fixed real number. For any $\gamma>0$
			and $B\geq 2$, let $\mathcal{N}(B,\gamma)$ denote the number of $4$-tuples $(n_1, n_2, n_3, n_4 )\in\{B + 1, B + 2,\dots, 2B\}^4$ for which
			\[|n_1^\theta+n_2^\theta-n_3^\theta-n_4^\theta|<\gamma.\]
			Then for every $\epsilon>0$
			\[\mathcal{N}(B,\gamma)\ll_\epsilon B^{2+\epsilon}+\gamma B^{4-\theta+\epsilon}.\]
		\end{theorem}	
		For $d$-dimensional sequence $(\boldx_n)$, and any fixed $\boldgama\in\R_{>0}^d$, similar to $\mathcal{N}(B,\gamma)$, we define
		\[\mathcal{N}(N,\boldgama):=\{\boldn\in[N,2N]^4\cap\N^4: |\boldx_{n_1}+\boldx_{n_2}-\boldx_{n_3}-\boldx_{n_4}|<\boldgama\}.\] 
		We consider a measure $d\boldmu_{\boldgama}$ on $\R^d$ by
		\begin{align*}
			d\boldmu_{\boldgama}(\boldx)=\prod_{1\leq i \leq d}\frac{\sin^2(\gamma_ix_i)}{\pi \gamma_ix_i^2}dx_i.
		\end{align*}
		Then, its Fourier transform is given by
		\begin{align}\label{hig dim fourier trf of mu}
			\widehat{\boldmu}_{\boldgama}(\boldt)=\int_{\R^d}e^{-i\boldt\cdot\boldx}d\boldmu(\boldx)=\prod_{i\leq d}\max\left(1-\frac{|t_i|}{2\gamma_i},0\right).
		\end{align}
		Let $L\geq 1$ be the smallest integer such that $N\leq 2^L.$ Then, by application of H\"older’s inequality and~\eqref{hig dim fourier trf of mu}, we have
		\begin{align}\label{add eng observation}
			E_{\boldsymbol{\gamma}}(\boldX_N^{D})&\ll\int_{\R^d}\left(\sum_{l=0}^{L}\sum_{2^{l-1}<n\leq 2^l}e\left(x_n^{(i)}\alpha_1+\cdots+x_n^{(d)}\alpha_d\right)\right)^4d\boldmu_{\boldsymbol{\gamma}}(\boldalpha)\nonumber\\
			&\ll \int_{\R^d}\left(\sum_{l=0}^{L}1\right)^3\sum_{l=0}^{L}\left|\sum_{2^{l-1}<n\leq 2^l}e\left(x_n^{(i)}\alpha_1+\cdots+x_n^{(d)}\alpha_d\right)\right|^4d\boldmu_{\boldsymbol{\gamma}}(\boldalpha)\nonumber\\
			&\ll(\log N)^3\sum_{l=0}^{L}\mathcal{N}(2^{l-1},2\boldsymbol{\gamma}).
		\end{align}
		To obtain an upper bound of the joint additive energy,
		it is enough to know the upper bound of $\mathcal{N}(N,\boldgama)$.	
		Since $E(\boldX_N^{D})\leq \min_{i\leq d}E(\boldX_N^{\{i\}}),$ from Theorem~\ref{RS theorem} and~\eqref{add eng observation}, we get an upper bound for the joint additive energy of $(n^{\theta_1}, \dots, n^{\theta_d})$ as
		\begin{align}\label{ntheta add eng}
			E(\boldX_N^{D})\ll N^{2+\epsilon}+N^{4-\max \theta_i+\epsilon},
		\end{align}
		for any $\epsilon>0.$
		
		Now we estimate the joint additive energy for $\boldX_N=\{(n,n\log^A n):n\leq N\}$.
		\begin{proposition}\label{lm1}
			Let $F$ be a twice continuously differentiable function such that $F''(n)\not\equiv 0$ and  $|F''(n)|\asymp |F''(N)|$ for $n\asymp N$. Let $\bcalX_N:=\{(n,F(n)):N\leq n\leq 2N\}$. For any $M>1$ we have 
			\begin{align*}
				E_\gamma(\bcalX_N)\ll\min\Big\{N^3, \Big(N^3/M+N^3|F''(N)|+N\log M/|F''(N)|\Big)\Big\}.
			\end{align*}
		\end{proposition}
		Choosing $F(x)=x(\log{x})^A$ and $M=N$ in Proposition \ref{lm1} we have $E_\gamma(\bcalX_N)\ll N^{2}(\log{N})^{A-1}+N^2(\log N)^{2-A}.$ Then, by~\eqref{add eng observation} we have the following:
		\begin{corollary}\label{boundcor}
			Let $A\geq1$ be a real number.
			\[E_\gamma(\boldX_N)\ll N^{2}(\log{N})^{A+2}+N^2(\log N)^{5-A}.\]
		\end{corollary}
		In order to prove the Proposition~\ref{lm1} we need the following two lemmas. 
		
		\begin{lemma}[Lemma 2.1 \cite{N Watt}]\label{energy to integral lemma}
			Let $K$ be a positive integer. Let $\mathcal{A}\subset\R$ be bounded and $\boldsymbol{\omega}$ be an $\R^K$-valued function defined on $\mathcal{A}\cap\Z$. Let $V_{2M}(\mathcal{A},\boldsymbol{\omega},\boldsymbol{\delta})$ be the number of solutions of the system of Diophantine inequalities,
			\begin{align}\label{diophantine inequalities}
				\left|\sum_{m=1}^{M}(\omega_k(u_m)-\omega_k(v_m))\right|<\delta_k \text{  for }k=1,\dots, K,
			\end{align}
			where $u_m,v_m\in\mathcal{A}\cap\Z$ for $m=1,\dots,M.$
			Let $2\delta_kD_k=1,$ for $k=1,\dots,K.$ Then
			\begin{align}\label{energy to integral}
				\delta_1\cdots\delta_KW_{2M}(\mathcal{A},\boldsymbol{\omega},\boldsymbol{\delta})\ll_{K}V_{2M}(\mathcal{A},\boldsymbol{\omega},\boldsymbol{\delta})\ll_{K}\delta_1\cdots\delta_KW_{2M}(\mathcal{A},\boldsymbol{\omega},\boldsymbol{\delta}),
			\end{align}
			where
			\begin{align*}
				W_{2M}(\mathcal{A},\boldsymbol{\omega},\boldsymbol{\delta})=\int_{-D_1}^{D_1}\cdots\int_{-D_K}^{D_K}|T(\mathcal{A},\boldsymbol{\omega},\boldx)|^{2M}dx_1\dots dx_K
			\end{align*}
			and
			\begin{align*}
				T(\mathcal{A},\boldsymbol{\omega},\boldx)=\sum_{u\in\mathcal{A}\cap\Z}e(\boldsymbol{\omega}(u)\cdot \boldx).
			\end{align*}
		\end{lemma}
		\begin{lemma}\cite[Theorem 2.2]{Graham Kolesnik}\label{vander corput thm}
			Suppose that $f$ is a real valued function with two continuous derivatives on an interval $I$. Suppose there exits some $\lambda >0$ and some $\alpha\geq1$ such that
			$$\lambda \leq|f''(x)|\leq\alpha\lambda$$ 
			on $I$. Then
			\[\sum_{n\in I}e(f(n))\ll\alpha|I|\lambda^{1/2}+\lambda^{-1/2}.\]
		\end{lemma}
		
		\begin{proof}[Proof of Proposition~\ref{lm1}]
			In order to estimate $E_\gamma(\bcalX_N)$, it is enough to estimate the count
			\begin{align*}
				\#\left\{\underline{n}\in[N,2N]^4:
				n_1+n_2=n_3+n_4,
				\left|F(n_1)+F(n_2)-F(n_3)-F(n_4)\right|<\gamma
				\right\}.
			\end{align*}
			If $n_1+n_2=n_3+n_4 =k$ then $k\in [2N, 4N]$, thus the count above is bounded by 
			\begin{align*}
				\#\left\{ k\in[2N,4N] \mbox{ and }\underline{n}\in[N,2N]^2 : \left|F(n_1)+F(k-n_1)-F(n_2)-F(k-n_2)\right|<\gamma \right\}.
			\end{align*}
			For $\delta>0$, $\mathcal{N}(N,M,\delta)$ to denote the counting 
			\begin{align*}
				\#\left\{k\in[2N,4N] \mbox{ and }\underline{n}\in[N,2N]^2 : \left|F(n_1)+F(k-n_1)-F(n_2)-F(k-n_2)\right|<\delta M\right\}.
			\end{align*}
			Then by Lemma~\ref{energy to integral lemma} it is not difficult to obtain
			\begin{align}\label{N}
				\mathcal{N}(N,M,\delta)\ll \delta\sum_{k\in[2N,4N]}\int_{0}^{1/2\delta}\bigg|\sum_{n\in[N,2N]}e\left(\left(F(n)+F(k-n)\right)y/M\right)\bigg|^2dy.
			\end{align}
			Let us call the integral $I(k,N,\delta)$ and the exponential sum $S(N,y)$. Observe that  whenever $k\in [2N, 4N]$ and $n\in [N, 2N]$, the size of the second order differentiation of the argument of the exponential function is of size $\asymp |F''(N)|y/M$.
			Thus by Lemma~\ref{vander corput thm},
			\begin{align}\label{s,F eqn}
				S(N,y)\ll N(|F''(N)|y/M)^{1/2}+(|F''(N)|y/M)^{-1/2}.
			\end{align}
			For $\delta=M^{-1}$ we write 
			\begin{align*}
				I(k,N,M^{-1}) \leq\int_{0}^{1}|S(N,y)|^2dy + \int_{1}^{M}|S(N,y)|^2dy.
			\end{align*}
			The first integral trivially bounded by $N^2$. For second integral, using~\eqref{s,F eqn} we get 
			\begin{align*}
				I(k,N,M^{-1})  \ll N^2 + N^2|F''(N)|M+M\log M/|F''(N)|.
			\end{align*}
			Replacing above estimate of $I(k,N,M^{-1})$ in \eqref{N} to deduce
			\begin{align*}
				\mathcal{N}(N,M,M^{-1})\ll N^3/M+N^3|F''(N)|+N\log M/|F''(N)|.
			\end{align*}
			Since $E_{\gamma}(\bcalX_N)\ll \mathcal{N}(k,N,M^{-1})$ and the trivial upper bound of $E_{\boldsymbol{\gamma}}(\bcalX_N)$ is $\ll N^3$, the result follows.
		\end{proof}
		\begin{remark}
			Using Lemma~\ref{energy to integral lemma} and Lemma~\ref{vander corput thm} it can be shown that for $A\geq 1$ the additive energy of $X_N=\{n\log ^{A}n: n\leq N\}$ satisfies 
			\begin{align}\label{additive of nlogn}
				E(X_N)\ll N^3(\log{N})^{2A+1}+N^3(\log N)^{5-A}.
			\end{align}
		\end{remark}
		We use the above estimate in corollary~\ref{n,logn PPC}.
		In the subsequent sections we prove Theorem~\ref{thmoc} and Theorem~\ref{thm1}. 
		
		\section{Preparation of the proofs of Theorem~\ref{thmoc} and Theorem~\ref{thm1}}
		For $\boldx$ in $\R^d$ we define the characteristic function $\chi_{s,N^{1/d}}$ by
		$$\chi_{s,N^{1/d}}(\boldx)= \begin{cases}
			1 & \mbox{ if }  \|\boldx\|_\infty^{\text{intdist}}\leq \frac{s}{N^{1/d}},\\
			0 & \mbox{ otherwise }.
		\end{cases}
		$$
		Thus, in order to prove Theorem~\ref{thm1} we have to show that for almost all $\boldalpha$ in $\R^d$
		\begin{align}\label{convergence}
			\frac{1}{N}\sum_{1\leq m\neq n\leq N}\chi_{s,N^{1/d}}((\boldx_n-\boldx_m)\boldalpha)
			\to (2s)^d \quad \mbox{ as } N\rightarrow \infty.
		\end{align}
		It is clear from the definition of $\chi_{s,N^{1/d}}$ that 
		\begin{align}\label{character}
			\chi_{s,N^{1/d}}(\boldx)=\chi^{1}_{s,N^{1/d}}(x^{(1)})\cdots\chi^{1}_{s,N^{1/d}}(x^{(d)}),
		\end{align}
		where, $\chi^{1}_{s,N^{1/d}}(x^{(i)})=1$ if  $\displaystyle\|x^{(i)}\|^{\text{(intdist)}}=\min_{k\in \mathbb{Z}}|x^{(i)}+k|\leq s/N^{1/d}$ and $0$ otherwise.
		It is well known that for any $s$ and $N$ there exist trigonometric polynomials $f_{K,s,N^{1/d}}^\pm(x)=\sum_{|j|\leq K}c_j^\pm e(jx)$, for a positive integer $K$, such that
		\begin{align}\label{majorant-minorant}
			f_{K,s,N^{1/d}}^-(x)\leq\chi^{1}_{s,N^{1/d}}(x)\leq f_{K,s,N^{1/d}}^+(x),
		\end{align}
		for all $x\in\R$. Further, we have
		\begin{align}\label{fouriercoeffequality}
			c_0^\pm=\int_{0}^{1}f_{K,s,N^{1/d}}^{\pm}(x) dx=\frac{2s}{N^{1/d}}\pm\frac{1}{K+1},
		\end{align}
		and any non-zero $j$-th Fourier coefficient $c_j^\pm$ of $f_{K,s,N^{1/d}}^\pm(x)$ satisfies the upper bound
		\begin{align}\label{fouriercoeffinequality}
			|c_j^\pm|\leq\min\left(\frac{2s}{N^{1/d}},\frac{1}{\pi|j|}\right)+\frac{1}{K+1}.
		\end{align}
		For more details one can see ~\cite[Chapter 1]{montgomery}. Thus, by imposing \eqref{majorant-minorant} in \eqref{character} we obtain
		\begin{align}\label{highinequality}
			F_{K,s,N^{1/d}}^-(\boldx)\leq\chi_{s,N^{1/d}}(\boldx)\leq F_{K,s,N^{1/d}}^+(\boldx),
		\end{align}
		where 
		\begin{align*}
			F_{K,s,N^{1/d}}^\pm(\boldx)&:=f_{K,s,N^{1/d}}^\pm(x^{(1)})\cdots f_{K,s,N^{1/d}}^\pm(x^{(d)})\\
			&=\displaystyle\sum_{\substack{\boldj\in\Z^d\\
					\|\boldj\|_\infty\leq K}}c_{\boldj}^\pm e(\boldj\cdot \boldx),
		\end{align*} 
		where 	$ c_{\boldj}:=c_{j_1}\cdots c_{j_d}$ and $c_{j_1},\dots,c_{j_d}$ satisfy ~\eqref{fouriercoeffequality} and ~\eqref{fouriercoeffinequality}.
		
		In order to establish \eqref{convergence} for $\chi_{s,N^{1/d}}$ it is enough to establish such result for its majorant and  minorant  $F_{K,s,N^{1/d}}^\pm$. Thus we want to show that for every fixed positive integer $r$ and for every $s>0,$
		\begin{align}\label{required}
			\frac{1}{N}\sum_{1\leq m\neq n\leq N}F_{(rN)^{1/d},s,N^{1/d}}^\pm((\boldx_n-\boldx_m)\boldalpha)\sim
			N\int_{[0,1]^d}F_{(rN)^{1/d},s,N^{1/d}}^\pm(\boldy)dy ~~~\mbox{ as } N \rightarrow \infty,
		\end{align}
		holds for almost all $\boldalpha$ in $\R^d.$ Then we can conclude \eqref{convergence} by using the fact that  $c_{\bold{0}}^{\pm}=(c_0^{\pm})^d$, \eqref{fouriercoeffequality} and tend $r$ to infinity.
		
		\subsection{Mean estimation}
		Now we prove ~\eqref{required} for $F_{K,s,N^{1/d}}^{+}$ using ``variance method". Similar arguments works for $F_{K,s,N^{1/d}}^{-}$.
		From now on $r$ and $s$ are fixed and we denote $F_{(rN)^{1/d},s,N^{1/d}}^+$ by $F_N$. 
		Let us consider the measure $$d\boldmu(\boldx)=\prod_{1\leq i \leq
			d}\frac{2\sin^2(x_i/2)}{\pi x_i^2}dx_i.$$
		Note that any $\boldmu$ almost sure result implies Lebesgue almost sure result and its Fourier transform is supported on $(-1,1)^d$ (see~\eqref{hig dim fourier trf of mu}). At first we want to compute the mean value of the quantity of interest as follows:
		\begin{align*}
			&\left|\int_{\R^d}\frac{1}{N}\sum_{1\leq m\neq n\leq N}F_N((\boldx_n-\boldx_m)\boldalpha)d\boldmu(\boldalpha)-N\int_{[0,1]^d}F_N(\boldy)dy\right|\\
			&\leq\left(N-\frac{N^2-N}{N}\right)\int_{[0,1]^d}F_N(\boldy)dy\\
			& +\frac{1}{N}\sum_{\substack{\boldj\in\Z^d\\
					1\leq\|\boldj\|_\infty\leq (rN)^{1/d}}}|c_{\boldj}| \left|\int_{\R^d}\sum_{1\leq m\neq n\leq N} e(\boldj\cdot (\boldx_n-\boldx_m)\boldalpha)d\boldmu(\boldalpha)\right|\\
			& \ll N^{-1}+N^{-2}\sum_{\substack{\boldj\in\Z^d\\
					1\leq\|\boldj\|_\infty\leq (rN)^{1/d}}}\sum_{1\leq m\neq n\leq N}\mathds{1}_{\|\boldj(\boldx_m-\boldx_n)\|_\infty<1}\ll N^{-1},
		\end{align*}
		since $|c_{\boldj}|\ll N^{-1}$ and we used the well-spacing of the sequence to get that the sum over $\boldj, m$ and $n$ is at most $N$.
		\subsection{Variance estimation}\label{variance subsection}
		Set
		\[H_N(\boldx):=F_N(\boldx)-\int_{[0,1]^d}F_N(\boldy)d\boldy=\sum_{\boldj\in\Z^d\setminus\{\boldsymbol{0}\}}c_{\boldj}e(\boldj\cdot \boldx),\]
		where $c_{\boldj} =0$ for $\|\boldj\|_\infty>(rN)^{1/d}$.
		We want to estimate the following quantity
		\[\text{Var}(H_N,\boldmu):=\int_{\R^d}\left(\frac{1}{N}\sum_{1\leq m\neq n\leq N}H_N((\boldx_n-\boldx_m)\boldalpha)\right)^2d\boldmu(\boldalpha).\]
		Let $D'\subseteq D$ with $|D'|=d'$ and $D'= \{ i_1,\ldots,i_{d'}  \}$ such that $i_1<\dots<i_{d'}$. Any element in such a $d'$-dimensional vector subspace of $\mathbb{R}^{d}$, which is isomorphic to $\mathbb{R}^{d'}$, is denoted by $\boldx^{D'}=(x^{(i_1)},\dots,x^{(i_{d'})})$. Thus
		\begin{align*}
			\text{Var}(H_N,\boldmu)	&=\frac{1}{N^2}\int_{\R^d} \left(\sum_{\substack{D'\subseteq D}}\sum_{1\leq m\neq n\leq N}\sideset{}{'}\sum_{\boldj\in\Z^{d'}}c_0^{d-d'}c_{\boldj }e(\boldj\cdot(\boldx_m^{D'}-\boldx_n^{D'})\boldalpha)\right)^2d\boldmu(\boldalpha)\nonumber\\
			&\ll\frac{1}{N^2}\int_{\R^d}\left(\sum_{\substack{D'\subseteq D}}\sum_{1\leq m\neq n\leq N}\sideset{}{'}\sum_{\boldj\in\Z^{d'}}c_0^{d-d'}c_{\boldj }e(\boldj\cdot(|\boldx_m^{D'}-\boldx_n^{D'}|)\boldalpha)\right)^2d\boldmu(\boldalpha),
		\end{align*}
		where $|\boldx_m^{D'}-\boldx_n^{D'}|$ means the component-wise absolute difference and $\sum^{'}$ means the sum over $ \boldj\in\Z^{d'}$ such that $1\leq j_l\leq (rN)^{1/d}$ for all $l\leq d'$. Let $U$ be the smallest integer such that $2^U\geq(rN)^{1/d}$.
		Now, by dyadic decomposition of the components of $ \boldj =(j_1,\ldots,j_{d'})\in\Z^{d'}$ and an application of the Cauchy-Schwarz inequality gives the upper bound of $\text{Var}(H_N,\boldmu)$ as
		\begin{align}\label{variance bound}	
			\ll\frac{1}{N^2}\int_{\R^d}&\left(\sum_{\substack{\|\boldu\|_\infty\leq U\\\boldu\in\N^d}}1\right)\sum_{\substack{\|\boldu\|_\infty\leq U\\\boldu\in\N^d}}\left|\sum_{\substack{D'\subseteq D}}\sum_{1\leq m\neq n\leq N}\sum_{\substack{(j_1,\ldots,j_{d'})\in\Z^{d'}\\2^{u_i-1}\leq|j_i|<2^{u_i}}}c_0^{d-d'}c_{\boldj }e(\boldj\cdot(|\boldx_m^{D'}-\boldx_n^{D'}|)\boldalpha)\right|^2d\boldmu(\boldalpha)\nonumber\\
			&\ll\frac{(\log N)^d}{N^4}\sum_{\substack{\|\boldu\|_\infty\leq U\\\boldu\in\N^d}}\sum_{\substack{D'\subseteq D}}\sum_{\substack{1\leq m\neq n\leq N\\1\leq k\neq l\leq N}}\sum_{\substack{\boldj,\boldt\in\N^{d'}\\2^{u_i-1}\leq j_i,t_i<2^{u_i}}}\mathds{1}_{\left\|\boldj(|\boldx_m^{D'}-\boldx_n^{D'}|)-\boldt(|\boldx_k^{D'}-\boldx_l^{D'}|)\right\|_\infty<1}.
		\end{align}	
		For fixed $\boldu \in \mathbb{N}^d$ and $D'\subseteq D$, define the multi-set \[\{\boldz_1^{D'},\dots,\boldz_M^{D'}\}=\{|\boldx_n^{D'}-\boldx_m^{D'}|:1\leq n\neq m\leq N\},\]
		where $M:=N^2-N$.
		In order to bound ~\eqref{variance bound}, it is enough to bound the following sum
		\begin{align}\label{mainineq}
			\sum_{\substack{\boldj,\boldt\in\N^{d'}\\2^{u_i-1}\leq j_i,t_i<2^{u_i}}}\sum_{1\leq m,n\leq M}\mathds{1}_{\left\|\boldj\boldz_m^{D'}-\boldt\boldz_n^{D'}\right\|_\infty <1}.
		\end{align}
		
		\section{Proof of Theorem~\ref{thm1}}\label{subsection bound of mainineq}
		To prove Theorem~\ref{thm1} we need to bound~\eqref{mainineq}.
		In the subsequent calculations we consider the subset $D'=\{1,\dots, d'\}\subseteq D,$ because for other subset of $D$ with cardinality $d'$ arguments are same. For simplicity, we use the notation $\boldz_n $ instead of $\boldz_n^{D'}\in\R^{d'}$ in the remaining proof.
		
		Let $\beta$ be a function defined by $\beta(2)=-1$ and $\beta(d)=+1$ for $d>2$. Further define $\Phi(y)=e^{-y^2/2}$, then its Fourier transform is $\widehat{\Phi}=\sqrt{2\pi}\Phi$. Also, assume that $\epsilon>0$ is fixed but sufficiently small real number.
		Now we divide the sum over $m,n$ into two cases based on the size of $\min(z^{(l)}_m,z^{(l)}_n)$, where $z^{(l)}_m$ and $z^{(l)}_n$ are the $l$-th  component of $\boldz_m$ and $\boldz_n$, respectively.
		\begin{case}
			$\min(z_m^{(l)},z_n^{(l)})< N^{1/d+\beta(d)\epsilon}$.
			In this case ~\eqref{mainineq} is bounded by,
			\begin{align}\label{bound of caes1}
				\sum_{\substack{\boldj,\boldt\in\N^{d'}\\2^{u_i-1}\leq j_i,t_i<2^{u_i}}}\sum_{\substack{1\leq m,n\leq M\\z_m^{(l)},z_n^{(l)}\leq 4N^{1/{d}+\beta(d)\epsilon}}}\mathds{1}_{\left\|\boldj\boldz_m-\boldt\boldz_n\right\|_\infty<1}.
			\end{align}			
			For fixed $\boldj,\boldz_m $ and $\boldz_n$ there are $\ll1$ many choices for $\boldt,$ because of the spacing assumption.
			One observe that there are $\ll N^{1+1/d+\beta(d)\epsilon}$ many $z_n$ which are $\leq 4N^{1/d+\beta(d)\epsilon}.$  In this case,  the upper bound of  ~\eqref{bound of caes1} is
			\begin{align} \label{case1}
				N^{2+2/d+2\beta(d)\epsilon}\prod_{l=1}^{d'}2^{u_l}\ll N^{2+2/d+d'/d+2\beta(d)\epsilon}\leq N^{3+2/d+2\beta(d)\epsilon}.
			\end{align}
		\end{case}

		\begin{case}
			$\min(z_m^{(l)},z_n^{(l)})\geq N^{1/d+\epsilon}$ for all $l\in D'.$\\
			For $\boldk\in\N^{d'}$,  the interval $[\boldk,\boldk+1 )$ means the product $[k_1,k_1+1)\times\ldots\times [k_{d'}, k_{d'}+1)$ and then define 
			\begin{align}\label{b}
				b_{\boldk}:=\sum_{n\leq N}\mathds{1}_{\boldz_n\in[\boldk,\boldk+1)}.
			\end{align}
			We split $\mathbb{N}^{d'}$ into the intervals $I_{\boldh}$, where $\boldh\in \mathbb{N}^{d'}\cup \{ \bold0\}$ and $I_{\boldh}$ is given by
			\begin{align*}
				I_{\boldh}:=\prod_{l=1}^{d'} \left[\left[\left(1+\frac{1}{T_l}\right)^{h_l}\right], \left[\left(1+\frac{1}{T_l}\right)^{h_l+1}\right]\right),
			\end{align*}
			where $T_l:=2^{u_l}N^{1/{d}+\epsilon/{2d}}$ for $1\leq l\leq d'$	and set 
			\begin{align}\label{a}
				a_{\boldh} :=\left(\sum_{\boldk\in I_{\boldh}}b_{\boldk}^2\right)^{1/2}.
			\end{align}
			Using the notation
			$\bold{y}_{\bold{T}}:= \left(\frac{y_1}{T_1},\ldots, \frac{y_{d'}}{T_{d'}}\right)$, we set 
			\begin{align*}
				\Phi\left(\boldy_{\bold{T}}\right):=\prod_{l=1}^{d'}\Phi\left(\frac{y_l}{T_l}\right).
			\end{align*}
			Then we introduce the function $P:\mathbb{R}^{d'}\rightarrow \mathbb{C}$, which is given by
			\begin{align}\label{P}
				P(\boldy)=\sum_{\boldh\geq 0}a_{\boldh}\prod_{l=1}^{d'}\left(1+\frac{1}{T_l}\right)^{ih_l y_l}.
			\end{align}
		\end{case}
		Similar to Lemma 2 in \cite{CA2021Real}, we also get its higher dimensional analogue  in the following lemma.
		\begin{lemma}\label{lemma P and additive}
			Let us set $T:=T_1 \ldots T_{d'}$ . Then
			\[\int_{\R^{d'}}|P(\boldy)|^2\Phi\left(\boldy_{\bold{T}}\right) d\boldy\ll TE(\boldX_N^{D'}).\]
		\end{lemma}
		\begin{proof}
			Similar to the proof of~\cite[Lemma 2]{CA2021Real}, expanding the square and apply the Fourier transform on $\Phi(\boldy)$,  we get
			\begin{align*}
				\int_{\R^{d'}}|P(\boldy)|^2\Phi\left(\boldy_{\bold{T}}\right) d\boldy \ll T\sum_{\bold{h_1},\bold{h_2} \in \mathbb{N}^{d'}\cup \{ \bold0\} } a_{\bold{h_1}}a_{\bold{h_2}}\prod_{l=1}^{d'}\widehat{\Phi}\left(\frac{h_1^{(l)}-h_2^{(l)}}{2}\right) \ll  T\sum_{\bold{h} \in \mathbb{N}^{d'}\cup \{ \bold0\} } a_{\bold{h}}^2.
			\end{align*}
			The last inequality uses the rapid decay of $\widehat{\Phi}$. Then the definition \eqref{a} gives 
			\begin{align*}
				\sum_{\bold{h} \in \mathbb{N}^{d'}\cup \{ \bold0\} }a_{\bold{h}}^2 \ll \sum_{\bold{k} \in \mathbb{N}^{d'} }b_{\boldk}^2\ll \sum_{\substack{1\leq m,n\leq M\\ |z_m^{(l)} - z_n^{(l)}|<1, \: l\in D'}}1 \leq E(\bold{X}_N^{D'}).
			\end{align*}
			Hence the result follows.
		\end{proof}
		\begin{lemma}\label{lemma z to a}
			\[\sum_{\substack{\boldj,\boldt\in\Z^{d'}\\2^{u_l-1}\leq j_l,t_l<2^{u_l},\: l\in D'}}\sum_{1\leq m,n\leq M}\mathds{1}_{\left\|\boldj\boldz_m-\boldt\boldz_n\right\|_\infty<1} 
			\ll\sum_{\substack{\boldj,\boldt\in\Z^{d'}\\2^{u_l-1}\leq j_l,t_l<2^{u_l},\: l\in D'}}\sum_{\substack{\boldg,\boldh\in \mathbb{N}^{d'}\cup \{ \bold0\}\\\left|\left(1+\frac{1}{T_l}\right)^{h_l-g_l}-\frac{t_l}{j_l}\right|\leq \frac{4}{T_l},\: l\in D'}}a_{\boldg}a_{\boldh}.\]
		\end{lemma}
		\begin{proof}
			For $d'=1$, the proof is essentially identical with the proof of~\cite[Lemma 3]{CA2021Real}. For $d'>1$, we apply the same method for $d'$ many components in order to get the result.
		\end{proof}
		
		Let $K:\mathbb{R}\rightarrow \mathbb{R}^{+}$ be a function defined by 
		\begin{align}\label{definition of K}
			K(\xi)=\frac{\sin^2\left(\left(1/d+\delta/{d}\right)\xi\log N\right)}{\pi \xi^2\left(1/d+\delta/{d}\right)\log N},
		\end{align}
		for some positive real parameter $\delta$ and positive integer $d$. Its Fourier transform is given by 
		\begin{align}\label{definition of K hat}
			\widehat{K}(v)=\max\left(1-\frac{v}{2\left(1/d+\delta/{d}\right)\log N}, 0\right).
		\end{align}
		
		\begin{lemma}\label{lemma discrete to smooth}
			\begin{align*}
				&\sum_{\substack{\boldj,\boldt\in\Z^{d'}\\2^{u_l-1}\leq j_l,t_l<2^{u_l},\:l\in D'}}\sum_{\substack{\boldg,\boldh\in \mathbb{N}^{d'}\cup \{ \bold0\}\\\left|\left(1+\frac{1}{T_l}\right)^{h_l-g_l}-\frac{t_l}{j_l}\right|,\: l \in D'}}a_{\boldg}a_{\boldh}\\
				&\ll\frac{2^{u_1}\cdots2^{u_{d'}}}{T}\int_{\R^{d'}}\sum_{\boldj,\boldt\in \mathbb{N}^{d'}}\prod_{l=1}^{d'}\frac{\widehat{K}(\log j_lt_l)}{(j_lt_l)^{1/2}}\left(\frac{j_l}{t_l}\right)^{iy_l}|P(\boldy)|^2\Phi\left(\boldy_{\bold{T}}\right) d\boldy.
			\end{align*}
		\end{lemma}
		\begin{proof}
			Note that $2^{u_1}\cdots2^{u_{d'} }\ll \prod_{l=1}^{d'}(j_lt_l)^{1/2}\ll 2^{u_1}\cdots2^{u_{d'}}$. This fact combining with the properties of $\Phi$ as well as $\widehat{K}$, one can obtain this lemma. Also one may see the proof of~\cite[Lemma 4]{CA2021Real} for one dimensional version.
		\end{proof}
		
		\begin{lemma}\label{lemma From K to additive}
			\begin{align*}
				&\frac{2^{u_1}\cdots2^{u_{d'}}}{T}\int_{\R^{d'}}\sum_{\boldj,\boldt\geq 1}\prod_{l=1}^{d'}\frac{\widehat{K}(\log j_lt_l)}{(j_lt_l)^{1/2}}\left(\frac{j_l}{t_l}\right)^{iy_l}|P(\boldy)|^2\Phi\left(\boldy_{\bold{T}}\right)d\boldy\\
				&\ll N^{4-\epsilon d'/4d}+N^{\frac{152d'-31}{89d}+\frac{52}{89}+5\epsilon}E(\boldX_N^{D'})^{76/89}.
			\end{align*}
		\end{lemma}
		Therefore, in Case $2$, from Lemma~\ref{lemma z to a}, Lemma~\ref{lemma discrete to smooth} and Lemma~\ref{lemma From K to additive} the upper bound of~\eqref{mainineq}
		\begin{align}\label{case 2 bound}
			N^{4-\epsilon d'/4d}+N^{\frac{152d'-31}{89d}+\frac{52}{89}+5\epsilon}E(\boldX_N^{D'})^{76/89}.
		\end{align}

		In order to prove Lemma~\ref{lemma From K to additive} we  we need the following technical lemma, which is  essentially the Lemma 5.3 of \cite{BT} or \cite[Lemma 6]{CA2021Real}.
		\begin{lemma}~\label{bridge lemma}
			Let $\sigma\in(-\infty,1)$ and let $F$ be a holomorphic function in the strip $y=\Im z\in[\sigma-2,0]$ such that 
			\[\sup_{\sigma-2\leq y\leq 0}|F(x+iy)|\leq\frac{1}{x^2+1}.\]
			Then for all $s=\sigma+it\in \mathbb{C}$ with $ t\neq 0,$ we have
			\[\sum_{j,t\geq 1}\frac{\widehat{K}(\log jt)}{j^st^{\bar{s}}}=\int_{\R}\zeta(s+iu)\overline{\zeta(s-iu)}F(u)du+2\pi\zeta(1-2it)F(is-i)+2\pi\zeta(1+2it)F(i\bar{s}-i).\]
		\end{lemma}
		\begin{proof}[Proof of Lemma~\ref{lemma From K to additive}]
			Let us set 
			\begin{align}\label{G}
				G(y):=\displaystyle\sum_{j,t\geq 1}\frac{\hat{K}(\log jt)}{(jt)^{1/2}}\left(\frac{j}{t}\right)^{iy},
			\end{align}		
			and 
			\[G(\boldy):=G(y_1)\cdots G(y_{d'})=\sum_{\boldj,\boldt\geq 1}\prod_{l=1}^{d'}\frac{\widehat{K}(\log j_lt_l)}{(j_lt_l)^{1/2}}\left(\frac{j_l}{t_l}\right)^{iy_l}.\]
			Applying Lemma~\ref{bridge lemma} on each $G(y_l)$, where $y_l\neq 0$, we write
			\[G(\boldy)=\prod_{l=1}^{d'}[D_1(y_l)+D_2(y_l)+D_3(y_l)],\] where
			\begin{align*}
				&D_1(y_l):=\int_{\R}\zeta(1/2+iy_l+iu_l)\zeta(1/2-iy_l+iu_l)K(u_l)du_l,\\
				&D_2(y_l):=2\pi\zeta(1-2iy_l)F(-y_l-i/2),\\
				&D_3(y_l):=2\pi\zeta(1+2iy_l)F(y_l-i/2).
			\end{align*}
			Then term by term product gives 
			\begin{align}\label{G expression}
				G(\boldy)=\sum_{\eta_1,\dots,\eta_{d'}\in\{1,2,3\}}D_{\eta_1}(y_1)\cdots D_{\eta_{d'}}(y_{d'}),\text{ where } y_l\neq 0,\: \forall l\in D'.
			\end{align}
			From the definition of $P$ in \eqref{P}, we deduce that
			\begin{align}\label{P bound}
				|P(\boldy)|^2\leq&\left(\sum_{\boldh\geq 0}a_{\boldh}\right)^2\ll\left(\sum_{\boldk\geq 	1}b_{\boldk}\right)^2\ll N^4,\:\: \forall\boldy\in\R^{d'}.
			\end{align}
			Further, we have
			\begin{align}
				&|G(y)|\leq \sum_{\log(jt)\leq 2(1/{d}+\delta/{d})\log N}\frac{1}{(jt)^{1/2}}
				\leq\sum_{jt\leq N^{2(1/d+\delta/{d})}}\frac{1}{(jt)^{1/2}}
				\ll N^{1/d+\delta/{d}}\label{G bound},\\
				&|K(\pm y-i/2)|\ll N^{1/d+\delta/d}/(y+1)^2\label{K bound}.
			\end{align}
			Let us denote
			\begin{align*}
				I=\int_{\R^{d'}}G(\boldy)|P(\boldy)|^2\Phi\left(\boldy_{\bold{T}}\right)d\boldy.
			\end{align*}
			Now we bound the integral $I$ by breaking the range of integrals in several parts as follows:	
			\begin{align}\label{partition of I}
				I=\sum_{\mathcal{D}\subseteq D'}\int_{\substack{|y_l|\leq1,\:l\in \mathcal{D}\\ |y_l|>1,\:l\in D'\setminus \mathcal{D}}} G(\boldy)|P(\boldy)|^2\Phi\left(\boldy_{\bold{T}}\right)d\boldy =\sum_{\mathcal{D}\subseteq D'}I_{\mathcal{D}}~~~~~~ \mbox{ (say).}
			\end{align}
			For $\mathcal{D}=D'$, by employing \eqref{P bound} and ~\eqref{G bound} we get $I_{D'}\ll N^{4+d'/d+\delta d'/d}.$\\
			When $\mathcal{D}=\phi,$ the empty set, using ~\eqref{G expression} we want to estimate the upper bound  of $I_{\phi}.$ Thus
			\begin{align*}
				I_{\phi} \ll& \sum_{\substack{\eta_i\in\{1,2,3\}\\ i=1,\ldots,d'}}\int_{|y_l|>1,\:l \in D'}|D_{\eta_1}(y_1)\cdots D_{\eta_{d'}}(y_{d'})||P(\boldy)|^2\Phi\left(\boldy_{\bold{T}}\right)d\boldy
				= \sum_{\substack{\eta_i\in\{1,2,3\}\\ i=1,\ldots,d'}}I_{\eta_1,\dots,\eta_{d'}} ~~~\mbox{ (say)}.
			\end{align*}
			Now we bound $I_{\eta_1,\dots,\eta_{d'}}$ in three sub cases depending on the values of $\eta_1,\dots,\eta_{d'}$ as follows:
			\begin{itemize}
				\item[(i)] $\eta_1=\cdots=\eta_{d'}=1,$
				\item[(ii)] $\eta_1,\ldots,\eta_{d'}\in\{2,3\},$
				\item[(iii)] for some $l,k\in D'$ $\eta_l=1$ and $ \eta_k=2$ or $3$.
			\end{itemize}
			\subsubsection*{Subcase 1:} Assume $\eta_1=\cdots=\eta_{d'}=1$.
			In this case, we want to estimate
			\begin{align*}
				I_{1,\dots,1}= \int_{|y_l|>1,\: l\in D'}|P(\boldy)|^2\Phi\left(\boldy_{\bold{T}}\right)\int_{\mathbb{R}^{d'}}\prod_{l=1}^{d'}|\zeta(1/2+iy_l+iu_l)\zeta(1/2-iy_l+iu_l)K(u_l)|d\boldu d\boldy.
			\end{align*}
			Let us denote the inner integral ( integral over $\mathbb{R}^{d'}$)	in the above inequality by $N_{d'}(\boldy,\bold{T})$. Then we rewrite it as
			\begin{align}\label{N(y,T)}
				N_{d'}(\boldy,\bold{T})= \sum_{\widetilde{D}\subseteq D'}\int_{\substack{|u_l|\leq T_l,\: l\in\widetilde{D}\\|u_l|> T_l, \: l\in D'\setminus\widetilde{D}}} \prod_{l=1}^{d'}|\zeta(1/2+iy_l+iu_l)\zeta(1/2-iy_l+iu_l)K(u_l)|d\boldu.
			\end{align}
			For $\widetilde{D} = \phi$, using the trivial convexity bound of Riemann zeta function, $|\zeta(1/2+it)|\ll|t|^{1/4}$ and the upper bound $K(u_l)\ll u_l^{-2}$, the contribution to $N_{d'}(\boldy,\bold{T})$ is 
			\begin{align*}
				\ll \prod_{l=1}^{d'}\left(|y_l|^{1/2}/T_l+1/T_l^{1/2}\right).
			\end{align*}
			Consequently the contribution to $I_{1,\dots,1}$ in this case is given by $\ll N^4 T^{1/2}$, which is negligible.
			
			Next, we treat the case $\widetilde{D} = D' $. In this case, from \eqref{N(y,T)} we see that the range of the integral of any variable $u_l$ is $[-T_l, T_l]$.  We would use the following moment estimate of Riemann zeta function due to Ivi\'c (see \cite[Theorem 8.3]{I}). 
			\begin{align}\label{zeta moment}
				\int_{0}^{T}|\zeta(1/2+it)|^{\frac{178}{13}}dt\ll T^{\frac{29}{13}+\epsilon}.
			\end{align}
			Let $A,B$ be two integers such that $\frac{1}{A}+\frac{1}{A}+\frac{1}{B}=1.$ Then by H\"older's inequality  on the integration over $y_l$'s, gives the upper bound
			\begin{align*}
				&\ll\int_{|u_l|\leq T_l, \: l\in D'}\prod_{l=1}^{d'}K(u_l)\left(\int_{\R}|\zeta(1/2+iy_l+iu_l)|^A\Phi(y_l/T_l)dy_l\right)^{1/A}\times\\
				&\left(\int_{\R}|\zeta(1/2-iy_l+iu_l)|^A\Phi(y_l/T_l)dy_l\right)^{1/A}\times|P(0)|^{2(1-1/B)}\left(\int_{\R^{d'}}|P(\boldy)|^2\Phi\left(\boldy_{\bold{T}}\right)d\boldy\right)^{1/B}d\boldu\\
				&\ll\prod_{l=1}^{d'}\left( \int_{0}^{T_l}|\zeta(1/2+it)|^{A}dt\right)^{2/A} \times N^{8/A}T^{1/B}E^*(\boldX_N^{D'})^{1/B},
			\end{align*} 
			where the last inequality uses the estimates $\int_{|u_l|\leq T_l}K(u_l)du_l\ll 1$ and Lemma~\ref{lemma P and additive}.Taking $A=178/13$ and employing  \eqref{zeta moment}  we get the contribution to the upper bound of $I_{1,\dots,1}$ 
			\begin{align*}
				\ll T^{105/89+\epsilon} N^{52/89}E(\boldX_N^{D'})^{76/89}.
			\end{align*}
			Now, the remaining cases are when $\widetilde{D}$ is any proper non-trivial subset of $D'$. The associate contribution to the upper bound of $I_{1,\ldots,1}$ is given by,
			\begin{align*}
				\ll&\int_{|u_l|\leq T_l,\: l\in \widetilde{D}}\prod_{l\in\widetilde{D}} |K(u_l)|\int_{\R^{d'}}P(\boldy)|^2\Phi\left(\boldy_{\bold{T}}\right)\times\\
				&\prod_{l\in D'\setminus \widetilde{D}}\bigg(\frac{|y_l|^{1/2}}{T_l}+ \frac{1}{T_l^{1/2}}\bigg)\times
				\prod_{l\in\widetilde{D}}|\zeta(1/2+iy_l+iu_l)\zeta(1/2-iy_l+iu_l)|d\boldy \prod_{l\in\widetilde{D}}du_l,
			\end{align*}	
			where we use the estimate like the case $\widetilde{D} = \phi $ and interchanges of integrals. By using Holder's inequality on the integration over $\boldy$ (similar to the case $\widetilde{D} = D' $), the above estimate 
			\begin{align*}	
				\ll &\int_{\substack{|u_l|\leq T_l,\\ l\in \widetilde{D}}}\prod_{l=1}^{d'}K(u_l)\left(\int_{\R^{d'}}\prod_{l\in\widetilde{D}}|\zeta(1/2+iy_l+iu_l)|^A\prod_{l\in D'\setminus \widetilde{D}}\bigg(\frac{|y_l|^{1/2}}{T_l}+ \frac{1}{T_l^{1/2}}\bigg)^{A/2}\Phi\Big(\frac{y_l}{T_l}\Big)dy_l\right)^{1/A}\times\\
				&\left(\int_{\R^{d'}}\prod_{l\in\widetilde{D}}|\zeta(1/2-iy_l+iu_l)|^A\prod_{l\in D'\setminus \widetilde{D}}\bigg(\frac{|y_l|^{1/2}}{T_l}+ \frac{1}{T_l^{1/2}}\bigg)^{A/2}\Phi\Big(\frac{y_l}{T_l}\Big)dy_l\right)^{1/A}\times\\
				&|P(0)|^{2(1-1/B)}\left(\int_{\R^{d'}}|P(\boldy)|^2\Phi\left(\boldy_{\bold{T}}\right)d\boldy\right)^{1/B}\prod_{l\in\widetilde{D}}du_l.
			\end{align*}
			Note that $\int_{|u_l|\leq T_l}K(u_l)du_l\ll 1$  for any $l\in D'$ and also we estimate
			\begin{align*}
				\int_{\R}\left(|y_l|^{1/2}/T_l+1/T_l^{1/2}\right)^{A/2}\Phi(y_l/T_l)dy_l\ll T_l^{1-A/4}\ll 1 \mbox{ if } A\geq 4.
			\end{align*}
			Employing these notes, Lemma~\ref{lemma P and additive} as well as choosing $A=178/13$ and hence using \eqref{zeta moment}, we reduce the above estimate to the following upper bound
			\begin{align*}
				\prod_{l\in D'\setminus 	\widetilde{D}}T_l^{2/A-1/2}\prod_{l\in\widetilde{D}}T_l^{2\theta/A+\epsilon}N^{8/A}T^{1/B}E(\boldX_N^{D'})^{1/B}\ll T^{105/89+\epsilon} N^{52/89}E(\boldX_N^{D'})^{76/89}.
			\end{align*}
			Combining the estimates for any subset $\widetilde{D}$ of $D'$,	we conclude that 
			\begin{align*}
				I_{1\dots1}\ll N^4T^{1/2}+ T^{105/89+\epsilon} N^{52/89}E(\boldX_N^{D'})^{76/89}.
			\end{align*}
			
			\subsubsection*{Subcase 2:} Let $\eta_1,\ldots,\eta_{d'}\in\{2,3\}$ .
			Using ~\eqref{P bound}, ~\eqref{K bound} and the bound $|\zeta(1\pm it)|\ll\log |t|$, 
			\begin{align*}
				I_{\eta_1,\dots,\eta_{d'}}\ll \int_{|y_l|>1,\: l\in D'}N^4N^{d'/d+\delta d'/d}\prod_{l=1}^{d'}\frac{\log y_l}{(y_l+1)^2}d\boldy
				\ll N^{4+d'/d+\delta d'/d}.
			\end{align*}
			
			\subsubsection*{Subcase 3:}There exist some $l,k\in D'$ such that $\eta_l=1$ and $ \eta_k=2$ or $3$. Let $D''$ be the subset of $D'$ such that $l\in D''$ imply $\eta_l=1$. Set $d''=|D''|\geq 1$, then $d'-d''= |D'\setminus D'' |\geq 1$ and hence $l \in D'\setminus D''$ means  $\eta_l \in \{2,3\}$. In this case the upper bound of $I_{\eta_1,\dots,\eta_{d'}}$ is 
			\begin{align*}
				\int_{|y_l|>1,\:l\leq d'}|P(\boldy)|^2\Phi\left(\boldy_{\bold{T}}\right)&\int_{\R^{d''}}\prod_{l\in D''}|\zeta(1/2+iy_l+iu_l)\zeta(1/2-iy_l+iu_l)K(u_l)|d\boldu \\
				&\times\prod_{l\in D'\setminus D''}|2\pi\zeta(1\pm2iy_l)K(\pm y_l-i/2)|d\boldy.
			\end{align*}
			Now the estimates ~\eqref{K bound} and $|\zeta(1\pm it)|\ll\log |t|$ gives
			\begin{align*}
				I_{\eta_1,\dots,\eta_{d'}} \ll N^{(1/d+\delta/d)(d'-d'')}\int_{|y_l|>1,\: l\in D'}|P(\boldy)|^2\Phi\left(\boldy_{\bold{T}}\right)\prod_{l\in D'\setminus D''}\frac{\log |y_l|}{y_l^2+1}N_{l\in D'' }(y_l, T_l) d\boldy,
			\end{align*}
			where $N_{l\in D'' }(y_l, T_l)$ is defined by
			\begin{align*}
				N_{l\in D'' }(y_l, T_l):= \int_{\R^{d''}}\prod_{l\in D''}|\zeta(1/2+iy_l+iu_l)\zeta(1/2-iy_l+iu_l)K(u_l)|d\boldu.
			\end{align*}
			Similar to $N_{d'}(\boldy, \bold{T})$  in \eqref{N(y,T)}, we break the integrals of $N_{l\in D'' }(y_l, T_l)$ as 
			\begin{align*}
				\sum_{\widetilde{D}'\subseteq D''}\int_{\substack{|u_l|\leq T_l,\: l\in\widetilde{D}'\\|u_l|> T_l,\: l\in D''\setminus\widetilde{D}'}} .
			\end{align*}
			The remaining computation of $I_{\eta_1,\dots,\eta_{d'}}$ in this case is very similar to the Subcase 1, where we computed  $I_{1,\dots, 1}$. Let $J_\phi,  J_{D''}$ and $J_{\widetilde{D}'}$ be the contributions to $I_{\eta_1,\dots,\eta_{d'}}$ according to the choices $\widetilde{D}'= \phi$, $\widetilde{D}'= D''$ and $\widetilde{D}'\neq \phi$ and $D''$, respectively. Then we get
			\begin{align*}
				J_\phi \ll N^{(1/d+\delta/d)(d'-d'')} N^4\prod_{l\in D''}T_l^{1/2}.
			\end{align*}
			Following Subcase 1, as an application of Holder's inequality with $\frac{1}{A}+\frac{1}{A}+\frac{1}{B}=1$, we obtain
			\begin{align*}
				J_{D''}, J_{\widetilde{D}'} \ll N^{(1/d+\delta/d)(d'-d'')} N^{8/A}T^{1/B}E(\boldX_N^{D'})^{1/B}\prod_{l\in D''}T_l^{2\theta/A+\epsilon},
			\end{align*}
			where  $A=178/13$ and $\theta=29/13$.
			Combining the above estimates we have 
			\begin{align*}
				I_{\eta_1,\dots,\eta_{d'}} \ll N^{4+ (1/d+\delta/d)(d'-d'')}\prod_{l\in D''}T_l^{1/2} + N^{52/89+d'/d-d''/d+\delta(d'-d'')/d} T^{76/89}E(\boldX_N^{D'})^{76/89}\prod_{l\in D''}T_l^{29/89+\epsilon}.
			\end{align*}
			Choose $\delta=\epsilon/4$ and we note that $2^{u_l}\ll N^{1/d}$. This gives $T_l=2^{u_l}N^{1/d+\epsilon/2d}\ll N^{2/d+\epsilon/2d}$ , for all $l\in D'$. Hence we have 
			\begin{align}\label{T}
				T\ll N^{2d'/d+d'\epsilon/2d} \mbox{ and } \prod_{l\in D''}T_l\ll N^{2d''/d+d''\epsilon/2d}.
			\end{align} 
			
			Combining the estimates from Subcese 1, 2 and 3, with $\delta=\epsilon/4$, and using \eqref{T}, we get
			\begin{align*}
				I_\phi \ll&N^{4+d'/d+\epsilon d'/4d} + 
				\max_{1\leq 	d''<d'}\Big\{ N^{-d''/d+241d'/89d+58d''/89d+5\epsilon}\Big\}N^{52/89} E(\boldX_N^{D'})^{76/89}.
			\end{align*}
			Since the coefficient of $d''/d$ in the exponent of $N$ is negative, the reasonable choice is $d''=1$. Thus,
			\begin{align}\label{bound of I_3}
				I_\phi\ll N^{4+d'/d+\epsilon 	d'/4d}+N^{52/89+241d'/89d-31/89d+5\epsilon}E(\boldX_N^{D'})^{76/89}.
			\end{align}
			
			The remaining case in \eqref{partition of I}, when $\mathcal{D}$ is a non trivial proper subset of $D'$. Assume that $|\mathcal{D}|=r$. Then by using ~\eqref{G bound}, it is easy to see that 
			\begin{align*}
				I_{\mathcal{D}} \ll N^{(1/d+\delta/d)r}\int_{\substack{|y_l|\leq1\:: \:l\in \mathcal{D}\\ |y_l|>1\::\:l\in D'\setminus \mathcal{D}}}|P(\boldy)|^2\Phi\left(\boldy_{\bold{T}}\right)\prod_{l\in D'\setminus \mathcal{D}}G(y_l)d\boldy.
			\end{align*}
			Then we see that the integral in the right hand side can be obtain in a similar way as the estimate of $I_{\phi}$ in \eqref{bound of I_3}. In particular, we get such bound with the replacement $d'$ by $d'-r$ in \eqref{bound of I_3}.
			Essentially the bound for $I_{\mathcal{D}}$, is given by 
			\begin{align*}
				I_{\mathcal{D}}  \ll N^{(1/d+\epsilon/4d)r }\Big(N^{4+(d'-r)/d+\epsilon(d'-r)/4d}+N^{52/89+241(d'-r)/89d-31/89d+5\epsilon}E(\boldX_N^{D'})^{76/89}\Big) \ll I_\phi.
			\end{align*}
			Above estimate and \eqref{bound of I_3} combining with \eqref{partition of I} we obtain 
			\begin{align}\label{Bound of I in N}
				I\ll N^{4+d'/d+\epsilon d'/4d}+N^{52/89+241d'/89d-31/89d+5\epsilon}E(\boldX_N^{D'})^{76/89}.
			\end{align}
			Therefore, \eqref{Bound of I in N} and the bound $2^{u_1}\cdots2^{u_{d'}}/T=1/N^{d'/d+d'\epsilon/2d}$ prove Lemma~\ref{lemma From K to additive}.
		\end{proof}
		
		For $d\geq 3,$ combining the bounds for ~\eqref{mainineq} from
		~\eqref{case1} in Case 1 and ~\eqref{case 2 bound} in Case 2, we get 
		\begin{align*}
			\sum_{\substack{\boldj,\boldt\in\N^{d'}\\2^{u_l-1}\leq j_l,t_l<2^{u_l}}}\sum_{1\leq m,n\leq M}\mathds{1}_{\left\|\boldj\boldz_m^{D'}-\boldt\boldz_n^{D'}\right\|_\infty <1}\ll 	N^{3+2/d+2\epsilon}+N^{4-\epsilon d'/4d}+N^{\frac{152d'-31}{89d}+\frac{52}{89}+5\epsilon}E(\boldX_N^{D'})^{76/89}.
		\end{align*}
		
		Replacing above estimate in ~\eqref{variance bound}  we deduce that
		\begin{align*}
			\text{Var}(H_N,\boldmu)&\ll(\log N)^{2d}\left(N^{-\epsilon d'/4d}+\max_{D'\subseteq D}\Big\{N^{\frac{152d'-31}{89d}+\frac{52}{89}-4+5\epsilon}E^*(\boldX_N^{D'})^{76/89} \Big\}\right) \ll N^{-\kappa},
		\end{align*}
		for some $\kappa>0,$ provided that the hypothesis on the joint additive energy \eqref{thm1 conditions} is true.
		
		To get an upper bound estimate of \eqref{mainineq} for d=2, breaking into two cases is not enough, we need to break in few more cases.  We already addressed the case $\min(z_m^{(l)},z_n^{(l)})< N^{1/d-\epsilon}$ for some $l=1,2$ in Case 1.
		The complement situation of this case can be break in four parts depending on the range of $\min(z_m^{(l)},z_n^{(l)})$.
		\begin{itemize}
			\item[(i)]$\min(z_m^{(l)},z_n^{(l)}) \in[N^{1/2-\epsilon},N^{1/2+2\epsilon})$ for $l=1,2$,
			\item[(ii)] $\min(z_m^{(1)},z_n^{(1)})\in[N^{1/2-\epsilon},N^{1/2+2\epsilon})$ and $\min(z_m^{(2)},z_n^{(2)})\geq N^{1/2+2\epsilon}$,
			\item[(iii)] $\min(z_m^{(2)},z_n^{(2)})\in[N^{1/2-\epsilon},N^{1/2+2\epsilon})$ and $ \min(z_m^{(1)},z_n^{(1)})\geq N^{1/2+2\epsilon}$,
			\item[(iv)] $  z_m^{(l)},z_n^{(l)}\geq N^{1/2+ 2\epsilon}$ for $l=1,2.$
		\end{itemize}
		Observe that $2^{u_l-1}\leq j_l,t_l<2^{u_l}$ and $|j_lz_m^{(l)}-j_lz_m^{(l)}|<1$ together imply $z_m^{(l)}/z_n^{(l)}\in[1/4,4].$
		So, $\min(z_m^{(l)},z_n^{(l)})\in[C,D]$ implies $z_m^{(l)},z_n^{(l)}\in[C/4,4D].$
		
		For the case (i), it is enough to consider $z_m^{(l)},z_n^{(l)}\in[N^{\beta_l}/4,8N^{\beta_l})$ as we can cover  $[N^{1/2-\epsilon},N^{1/2+2\epsilon})$ by $\ll\log N$ many intervals of the form $[N^\beta_l,2N^\beta_l)$, where $1/2-\epsilon\leq\beta_l\leq1/2+2\epsilon.$ Then we can make the setup similar to Case 2 but here we define $T_l=2^{u_l-1}N^{\beta_l},\:l=1,2$ and $b_{\boldk}$ as follows
		\[b_{\boldk}=\sum_{\substack{1\leq m\leq M\\z_m^{(l)}\in[N^{\beta_l}/4,8N^{\beta_l}),\:l=1,2}}\mathds{1}(\boldz_m\in[\boldk,\boldk+1)).\]
		By Cauchy-Schwarz inequality we get $|P(0)|^2 \ll E(\boldX_N)N^{\beta_1+\beta_2}$ , which can use in the place of trivial upper bound $|P(0)|^2 \ll N^4$. Then we follow the Case 2 to bound~\eqref{mainineq}.
		
		For case (ii), we consider $z_n^{(1)},z_m^{(1)}\in[N^{\beta}/4,8N^{\beta})$ and $z_m^{(2)},z_n^{(2)}>N^{1/2+2\epsilon}$ where $1/2-\epsilon\leq\beta\leq1/2+2\epsilon.$ In this case we define $T_1=2^{u_1-1}N^\beta$ and $T_2=2^{u_2-1}N^{1/2+2\epsilon}$ and $b_{\boldk}$ as follows
		\[b_{\boldk}=\sum_{\substack{1\leq m\leq M\\z_m^{(1)}\in[N^{\beta}/4,8N^{\beta})\\z_m^{(2)}>N^{1/2+2\epsilon}}}\mathds{1}(\boldz_m\in[\boldk,\boldk+1)).\] 
		Here we have to work with the trivial bound $|P(0)|^2\ll N^4.$ Then following Case 2, we obtain bound of~\eqref{mainineq}. The case (iii) is same with case (ii), only we have to interchange the role of the components.
		The case (iv) essentially follow the Case 2 with using the definition $T_l=2^{u_l}N^{1/2+2\epsilon}$ for $l=1, 2$.
		In all first three cases, we can choose $\delta<\epsilon$ and for case (iv) take $\delta=\epsilon/2$ in the definition of function $K$ in \eqref{definition of K}. 
		Finally, combining all the estimates in these cases for \eqref{mainineq} and then replacing in ~\eqref{variance bound}, under the hypothesis on joint additive energy in the theorem we get $\text{Var}(H_N,\boldmu) \ll N^{-\kappa}$, for some $\kappa>0$.
		
		This is the main part of the proof. The rest of the arguments follow a standard method of applying Chebyshev’s inequality and Borel-cantelli lemma (see the proof of Theorem 1 in \cite{CA2017additive} and also \cite{hinrichs2019multi,RZ1999}).
		
		\section{Proof of Theorem~\ref{thmoc} }	
		
		When $x_n^{(1)}=\cdots=x_n^{(d)} :=x_n$ (say), the $\text{Var}(H_N,\boldmu)$(see subsection~\ref{variance subsection}) is bounded above by
		\begin{align}\label{same component varinace bound}
			\frac{(\log N)^d}{N^4}\sum_{\substack{\|\boldu\|_\infty\leq U\\\boldu\in\N^d}}\sum_{\substack{\boldj,\boldt\in\N^d\\2^{\boldu-1}\leq\boldj,\boldt<2^{\boldu}}}\sum_{1\leq m,n\leq M}\mathds{1}_{\left\|\boldj z_m-\boldt z_n\right\|_\infty<1},
		\end{align}
		where the multi-set $\{z_1,\dots,z_M\}=\{|x_m-x_n|:1\leq n\neq m\leq N\}.$
		We have $d\geq 3$ and $\boldu \in \mathbb{N}^d$ be fixed in \eqref{same component varinace bound}. We divide the sum over $m,n$ into three cases depending on the size of $\min(z_m,z_n).$\\
		\textbf{Case 1:} $\min(z_m,z_n)<N^{1/d+\epsilon}$.\\ 
		This case is same as to the Case 1 of section~\ref{subsection bound of mainineq}.\\
		\textbf{Case 2:} $N^{1/d+\epsilon}\leq \min(z_m,z_n)< N^{1+\epsilon}$\\
		\textbf{Case 3:} $\min(z_m,z_n)\geq N^{1+\epsilon}$\\
		In Case 2 as well as in Case 3 we observe that for all $1\leq l,\nu\leq d$,
		\begin{align*}
			\left|\frac{j_l}{t_l}-\frac{z_n}{z_m}\right|<\frac{1}{z_mt_l} 
			\implies |j_lt_\nu-j_\nu t_l|<\frac{t_l+t_\nu}{z_m}\leq \frac{4(rN)^{1/d}}{N^{1/d+\epsilon}}.
		\end{align*}
		As a conclusion we have for sufficiently large $N$ and for all $1\leq l,\nu\leq d$, the equality ${j_l}/{t_l}={j_\nu}/{t_\nu}$.
		So in Case 2 and Case 3,  the estimate  for the sums  over  $m,n$  and $\boldj,\boldt$  in ~\eqref{same component varinace bound} is bounded above by
		\begin{align}\label{same component}
			&\sum_{\substack{\boldj,\boldt\in\N^d\\2^{\boldu-1}\leq\boldj,\boldt<2^{\boldu}\\j_1/t_1=\cdots = j_d/t_d}}\sum_{1\leq m,n\leq M}\mathds{1}_{\left\|\boldj z_m-\boldt z_n\right\|_\infty<1}.
		\end{align}
		Similar to the estimate of \eqref{mainineq} in the proof of Theorem \ref{thm1}	we have to bound~\eqref{same component} for Case 2 and Case 3 here.
		
		For Case 3, let $u=\max(u_1\dots,u_d)=u_\nu$ (say), set $T=2^{u-1}N^{1+\epsilon/2}$ and recall $b_k,a_h$ from \eqref{b} and \eqref{a} for dimension one. Then~\eqref{same component} is bounded above by
		\begin{align}\label{reduction}
			\sum_{2^{u-1}\leq j_\nu,t_\nu<2^{u}}\sum_{1\leq m,n\leq M}\mathds{1}_{\left|j_\nu z_m-t_\nu z_n\right|<1} \sum_{\substack{\boldj,\boldt\in\N^{d-1}\\2^{\boldu-1}\leq\boldj,\boldt<2^{\boldu}\\j_1/t_1=\cdots = j_d/t_d}}1.
		\end{align}
		The condition $j_1/t_1=\cdots = j_d/t_d$ gives  estimate for the sum over $\boldj,\boldt \in \N^{d-1}$ is $ \ll (N^{1/d})^{d-1}$  as $2^u \ll N^{1/d}$. The remaining sums have been estimated in Lemma 3, Lemma 4 and Lemma 5 of \cite{CA2021Real}, where they choose $2^u\ll N$. But here $2^u \ll N^{1/d}$ and thus we use the function $K$ and $G$ defined in \eqref{definition of K} and \eqref{G}, which depend on $d$, in their Lemma 4 and Lemma 5.
		Specifically, we  use the estimate from the first passage of page no. 499 in \cite{CA2021Real} to achieve the upper bound estimate for  \eqref{same component} as
		\begin{align*}
			&\ll N^{(1-1/d)}\frac{2^u}{T}\left(N^{4+1/d+\epsilon/4}+ T^{105/89+\epsilon}N^{52/89}E(X_N)^{76/89}\right)\\
			&\ll N^{4-\epsilon/4}+N^{157/89+16/89d+2\epsilon}E(X_N)^{76/89},
		\end{align*}	
		since $T=2^{u-1}N^{1+\epsilon/2}\ll N^{1+1/d+\epsilon/2}.$
		
		Now note that $2^{u-1}\leq j_\nu,t_\nu<2^{u}$ and $|j_\nu z_m-t_\nu z_n|<1$ together imply $z_m/z_n\in[1/4,4].$ Therefore, if $\min(z_m,z_n)\in[C,D)$ then $z_m,z_n\in[C/4,4D].$ 
		
		By the above observation, in Case 2, it is enough to consider $z_m,z_n \in [N^\beta/4, 8N^\beta)$ for some $\beta \in [1/d+\epsilon, 1+\epsilon]$. As we can cover the interval $[N^{1/d+\epsilon},N^{1+\epsilon})$ by $\ll\log N$ many sub-intervals of the form $[N^\beta,2N^\beta)$. Let $\beta$ be fixed and $T=2^{u-1}N^\beta.$  Then similar to \eqref{reduction}, the estimate for \eqref{same component} is given by
		\begin{align*}
			\ll (\log{N }) (N^{1/d})^{d-1} \sum_{2^{u-1}\leq j_\nu,t_\nu<2^{u}} \sum_{\substack{1\leq m,n\leq M\\ z_m,z_n\in[N^\beta/4, 8N^\beta]}}\mathds{1}_{\left|j_\nu z_m-t_\nu z_n\right|<1}.
		\end{align*}
		
		Then following case 2 in page no. 500 of \cite{CA2021Real} we get
		\begin{align*}
			\sum_{2^{u-1}\leq j_\nu,t_\nu<2^{u}} \sum_{\substack{1\leq m,n\leq M\\ z_m,z_n\in[N^\beta/4, 8N^\beta]}}\mathds{1}_{\left|j_\nu z_m-t_\nu z_n\right|<1}
			\ll\frac{2^u}{T}\left(E(X_N)N^{1/d+\beta+\epsilon/4}+E(X_N)N^{13\beta/89}T^{105/89+\epsilon}\right).
		\end{align*}
		Since $T=2^{u-1}N^\beta \ll N^{1/d+\beta}$, in Case 2, the upper bound estimate of \eqref{same component} is 
		\begin{align*}
			&\ll  \log{N } N^{1-1/d-\beta}\left(E(X_N)N^{1/d+\beta+\epsilon/4}+E(X_N)N^{118\beta/89+105/89d+(1+\beta)\epsilon}\right)\\
			&\ll E(X_N)N^{118/89+16/89d+3\epsilon}.
		\end{align*}		
		Combining the estimates in all three cases, the variance i.e., the upper bound of ~\eqref{same component varinace bound} is
		\begin{align*}
			\ll N^{-1+2/d+3\epsilon}+E(X_N)N^{-238/89+16/89d+4\epsilon}+N^{-\epsilon/8}+N^{-199/89+16/89d+2\epsilon}E(X_N)^{76/89}\ll N^{-\epsilon'},
		\end{align*}
		provided that the hypothesis on the additive energy  bound is true. The rest of the arguments follow a standard method of applying Chebyshev’s inequality and Borel-cantelli lemma (see the proof of Theorem 1 in \cite{CA2017additive} and also \cite{hinrichs2019multi,RZ1999}).
		\section*{Acknowledgement}
		The authors would like to express sincere appreciation to Professor Christoph Aistleitner for several insightful suggestions.
		Also, they  thank Professor Olivier Ramar\'e for some helpful suggestions.
		
	\end{document}